\documentclass[11pt]{article}
\usepackage{latexsym,amsmath,stackrel,color,amsthm,amssymb,epsfig,graphicx,mathrsfs}
\usepackage{graphicx}
\usepackage{amssymb}
\usepackage[left=1in,top=1in,right=1in,bottom=1in]{geometry}
\usepackage[linktocpage=true]{hyperref}
\usepackage{setspace}
\usepackage{amssymb, amsmath, amsthm, graphicx,mathrsfs}
\usepackage{caption}

\usepackage{tikz}
\makeatletter
\def\thm@space@setup{%
  \thm@preskip=\parskip \thm@postskip=0pt
}
\makeatother
\def\qed{\hfill\ifhmode\unskip\nobreak\fi\quad\ifmmode\Box\else\hfill$\Box$\fi}
\def\ite#1{\hfill\break${}$\hbox to 50pt {\quad(#1)\hfill}}

\newtheorem{thm}{Theorem}[section]

\newtheorem{definition}{Definition}

\newtheorem{lem}[thm]{Lemma}
\newtheorem{conjecture}{Conjecture}

\newtheorem{prop}{Proposition}[section]

\def\ex{{\rm{ex}}}
\def\c{{\circlearrowright}}

\newcommand{\vb}[1]{\boldsymbol{#1}}

\newcommand{\cir}{\circlearrowright}

\begin{document}

\pagestyle{myheadings}
\markright{{\small{\sc F\"uredi, Jiang, Kostochka, Mubayi, and Verstra\"ete:   Crossing paths
}}}

\title{\vspace{-1in} Extremal problems for convex geometric hypergraphs
	and ordered hypergraphs}

\author{
\hspace{0.8in} Zolt\'an F\" uredi\thanks{Research supported by grant K116769
from the National Research, Development and Innovation Office NKFIH and
by the Simons Foundation Collaboration grant \#317487.}
\and
Tao Jiang\thanks{Research partially supported by National Science Foundation award DMS-1400249.}
\and
Alexandr Kostochka\thanks{Research  supported in part by NSF grant
 DMS-1600592, by Award RB17164 of the UIUC Campus Research Board and by grants 18-01-00353A and 19-01-00682
  of the Russian Foundation for Basic Research.
} \hspace{0.8in} \smallskip \and
Dhruv Mubayi\thanks{Research partially supported by NSF awards DMS-1300138 and 1763317.} \and Jacques Verstra\"ete\thanks{Research supported by NSF award DMS-1556524.}
}

\maketitle

\vspace{-0.5in}

\begin{abstract}
An ordered hypergraph is a hypergraph whose vertex set is linearly ordered, and a convex geometric hypergraph is
a hypergraph whose vertex set is cyclically ordered. Extremal problems for ordered and convex geometric graphs have a rich history with applications to a variety of problems in combinatorial geometry. In this paper, we consider analogous extremal problems for uniform hypergraphs, and determine the order of magnitude
of the extremal function for various ordered and convex geometric paths and matchings.
Our results  generalize earlier works of Bra{\ss}-K\'{a}rolyi-Valtr, Capoyleas-Pach and Aronov-Dujmovi\v{c}-Morin-Ooms-da Silveira. We also provide a new generalization of the Erd\H os-Ko-Rado theorem in the ordered setting.
\end{abstract}

\section{Introduction}

An {\em ordered graph} is a graph together with a linear ordering of its vertex set. Extremal problems for ordered graphs have a long history, and were studied extensively in papers by Pach and Tardos~\cite{PT}, Tardos~\cite{T} and Kor\'{a}ndi, Tardos, Tomon and Weidert~\cite{KTTW}. Let $\ex_{\rightarrow}(n,F)$ denote the maximum number of edges in an $n$-vertex ordered graph that does not contain the ordered graph $F$. This extremal problem is phrased in~\cite{KTTW} in terms of pattern-avoiding matrices. Marcus and Tardos~\cite{MT} showed that if the forbidden pattern is
a permutation matrix, then the answer is in fact linear in $n$, and thereby solved the Stanley-Wilf Conjecture,
as well as a number of other well-known open problems. A central open problem in the area was posed by
Pach and Tardos~\cite{PT}, in the form of the following conjecture. An ordered graph has  interval chromatic  number two if it is bipartite with bipartition $A \cup B$ and $A$ precedes $B$ in the ordering of the vertices.

\begin{conjecture}\label{ptc}
Let $F$ be an ordered acyclic graph with interval chromatic number two. Then $\ex_{\rightarrow}(n,F) = O(n\cdot \mbox{\rm polylog} \,n)$.
\end{conjecture}

In support of Conjecture A, Kor\'{a}ndi, Tardos, Tomon and Weidert~\cite{KTTW} proved for a wide class of forests $F$ that
$\ex_{\rightarrow}(n,F) = n^{1 + o(1)}$. This conjecture is related to a question of Bra{\ss} in the context of
convex geometric graphs.

A {\em convex geometric (cg) graph} is a graph together with a cyclic ordering of its vertex set. Given a convex geometric graph $F$, let $\ex_{\circlearrowright}(n,F)$ denote the maximum number of edges in an $n$-vertex convex geometric graph that does not contain $F$. Extremal problems for geometric graphs have a fairly long history, going back to theorems on disjoint line segments~\cite{Hopf-Pannwitz,Sutherland,Kupitz-Perles}, and more recent results on crossing matchings~\cite{Brass-Karolyi-Valtr,Capoyleas-Pach}. Motivated by the famous Erd\H{o}s unit distance problem, the first author~\cite{Furedi} showed that the maximum number of unit
distances between points of a convex $n$-gon is $O(n\log n)$. In the vein of Conjecture \ref{ptc},
Bra{\ss}~\cite{Brass} asked for the determination of all acyclic graphs $F$ such that $\ex_{\circlearrowright}(n,F)$ is linear in $n$, and this problem remains open (recently it was solved for trees~\cite{FKMV}).

In this paper, we study extremal problems for ordered and convex geometric uniform hypergraphs. An {\em ordered} ({\em convex geometric}) $r$-graph is an $r$-uniform hypergraph whose vertex set is linearly (cyclically) ordered. Although the theory of cg (hyper)graphs can be studied independently of any geometric context,  extremal problems for both cg graphs and  hypergraphs are frequently motivated by problems in discrete geometry~\cite{Brass-Rote-Swanepoel, Pach-Pinchasi,Brass,Aronov}. Instances of the extremal problem for two disjoint triangles in the convex geometric setting are connected to the well-known triangle-removal problem~\cite{Gowers-Long}. In~\cite{FJKMV} we show that certain types of paths in the convex geometric setting give the current best bounds for the notorious extremal problem for tight paths in uniform hypergraphs.

One of the goals of this paper is to  study extremal problems simultaneously in the ordered and cg settings and compare and contrast their behaviors.

\section{Results}
We denote by $\ex_{\rightarrow}(n,F)$ ($\ex_{\circlearrowright}(n,F)$) the maximum number of edges in an $n$-vertex ordered (cg) $r$-graph that does not contain $F$, and let $\ex(n,F)$ denote the usual (unordered) extremal function.
Let $P$ be the linearly ordered path with three edges with ordered vertex set $1<2<3<4$ and edge set   $\{13, 32, 24\}$. In the convex geometric setting we use $P$ to denote the unique  cg graph isomorphic to the path with three edges where the edges $13$ and $24$ cross. We then have
\begin{equation} \label{=}\ex_{\to}(n, P) = 2n - 3=\ex_{\cir}(n,P) \qquad \hbox{ for $n \ge 3$ }\end{equation}
where the former is a folklore result and the latter is due to Bra{\ss}, K\'{a}rolyi and Valtr~\cite{Brass-Karolyi-Valtr}.
To our knowledge,  (\ref{=}) are the only known nontrivial exact results for connected ordered or convex geometric graphs
that  have crossings in their embedding.
%To our knowledge (\ref{=}) are the only known nontrivial exact results for (connected) linearly ordered and convex geometric graphs (or hypergraphs).
 These two simple exact results therefore provide a good launchpad for further investigation  in the hypergraph case. This is the direction we take, extending (\ref{=}) to longer paths and to the hypergraph setting.
 %Our main tools are Theorems~\ref{splitting1} and \ref{splitting}.
  In the process, we will also discover some subtle differences between the  ordered and convex geometric cases which are not visible in (\ref{=}).

There are many ways to extend the definition of a path to hypergraphs and we choose one of the most natural ones, namely tight paths. There are also many possibilities for the ordering of the vertices of the path  and again we make a rather natural choice, namely crossing paths which are defined below (a similar notion was studied by Capoyleas and Pach~\cite{Capoyleas-Pach} who considered the corresponding question for matchings in a cg graph).

A {\em tight $k$-path} is an $r$-graph whose edges have the form
$\{v_i,v_{i + 1},\dots,v_{i + r - 1}\}$ for $0 \leq i < k$. Typically, we list the vertices $v_0v_1\dots v_{k+r-2}$ in a tight $k$-path.
We let $<$ denote the underlying ordering of the vertices of an ordered  hypergraph. In the case of convex geometric hypergraphs, we slightly abuse the same notation so that
 $u_1<u_2<\cdots < u_{\ell}$ is shorthand for $u_1<u_2<\cdots < u_{\ell} < u_1$ which  means that   moving clockwise in the cyclic ordering of the vertices from $u_1$ we first encounter $u_2$, then  $u_3$, and so on until we finally encounter $u_{\ell}$ and then $u_1$ again. In other words, $u_1, \ldots, u_{\ell}$ is a cyclic interval where the vertices are listed in clockwise order. When needed, we use the notation $\vb{\Omega}_n$ to denote the vertex set of a generic $n$-vertex convex geometric hypergraph, with the clockwise ordering of the vertices.

\begin{definition} [Crossing paths in ordered and convex geometric hypergraphs] \label{defCP} An {\em $r$-uniform crossing $k$-path} $P_k^r$ in an ordered or convex geometric hypergraph is a tight $k$-path $v_0v_1\dots v_{r+k-2}$
with the ordering
\vspace{-0.1in}
\begin{center}
\begin{tabular}{lp{5.8in}}
{\rm (i)} & $v_0 < v_1 < v_2 < \dots < v_{r-1}$,   \\
{\rm (ii)} &  $v_j  <  v_{j+r}  < v_{j + 2r}  <  \cdots <  v_{j+1}$ for $j<r-1$ and\\
{\rm (iii)} & $v_0 < v_{r-1} < v_{2r-1} < v_{3r-1}<\cdots <v_{\left\lfloor \frac{r+k-2}{r}\right\rfloor r -1}$.
\end{tabular}
\end{center}
\end{definition}

An  ordered $P_5^2$  (Figure 1) and  a convex geometric  $P_7^2$ and $P_5^3$  (Figure 2) are shown below.

\begin{figure}[!ht]
	\begin{center}
		\includegraphics[width=3.5in]{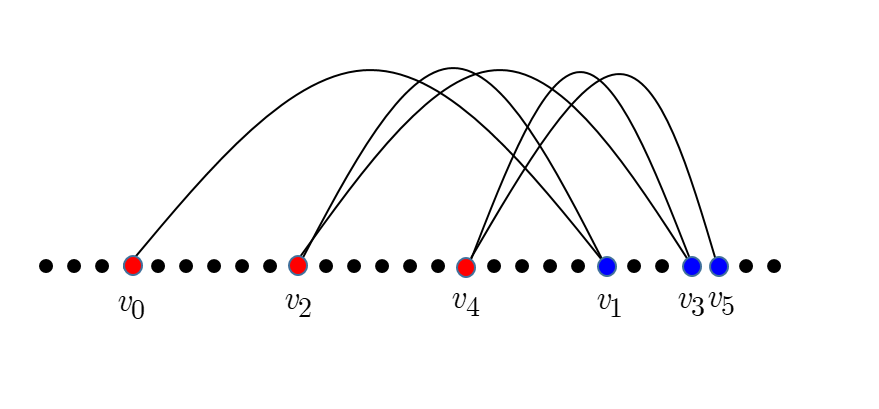}
		\caption{Ordered $P_5^2$}
		\label{fig:pathcrossers}
	\end{center}
\end{figure}

\begin{figure}[!ht]
\begin{center}
\includegraphics[width=3.5in]{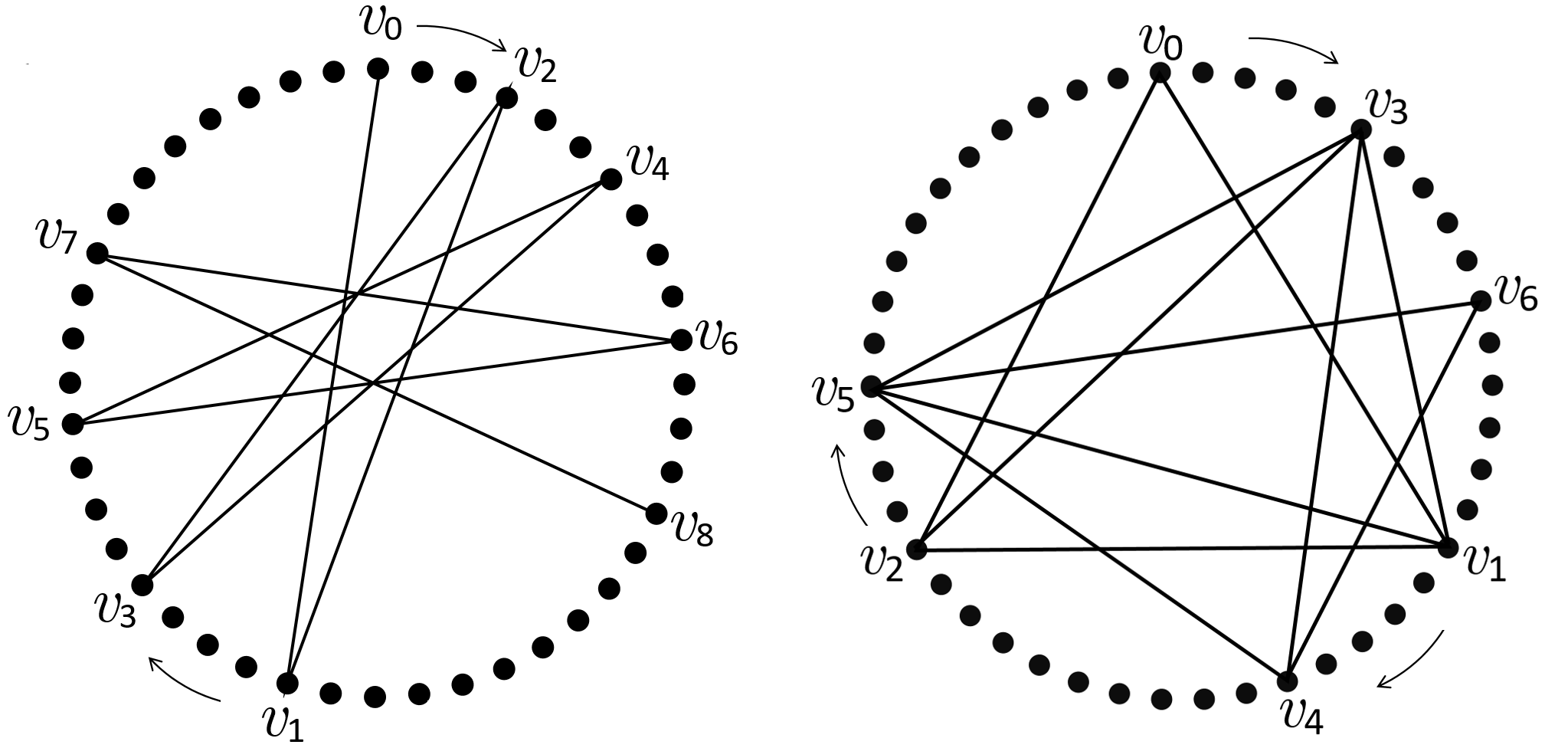}
\caption{Convex Geometric $P_7^2$ and $P_5^3$}
\label{fig:crossers}
\end{center}
\end{figure}
Our first result  generalizes $\ex_{\to}(n, P_3^2) = 2n - 3$ to larger $k$ and $r$.

\begin{thm}\label{cpthm}
Fix $k \geq 1$, $r \geq 2$ and let  $n\geq r+k$. Then
$$\ex_{\to}(n, P^r_k)=  \begin{cases}
{n \choose r} - {n-k+1 \choose r} &\mbox{ for } k \leq r + 1  \\
\Theta(n^{r - 1}\log n) & \mbox{ for }k \geq r + 2.
\end{cases}$$

\end{thm}

%Theorem \ref{cpthm} for $k \geq r + 2$ shows that the $\log n$ factor in Theorem %\ref{splitting1} is necessary, as we shall see for all $k,r \geq 2$ that %$z_{\rightarrow}(n,CP_k^r) = O(n^{r-1})$.

Our second theorem generalizes the Bra{\ss}, K\'{a}rolyi and Valtr~\cite{Brass-Karolyi-Valtr} result $\ex_{\cir}(n, P_3^2) = 2n - 3$ to larger $k$ and $r$.

\begin{thm}\label{cgthm}
Fix $k \geq 1$, $r \geq 2$ and let $n \ge 2r+1$. Then
  $$ \ex_{\circlearrowright}(n, P^r_k) = \begin{cases}
  \Theta(n^{r - 1}) & \mbox{ for } 3 \le k \leq 2r-1 \\
  {n \choose r} - {n - r \choose r} & \mbox{ for }k = r + 1 \\
  \Theta(n^{r - 1}\log n) & \mbox{ for }k \geq 2r.
  \end{cases}$$

  \end{thm}

For short paths  we have the following better bounds, which improve the previous results on this problem by Aronov et. al.~\cite{Aronov} when $k=2$.

\begin{thm}\label{maincrossing_short}
	For fixed $2\le k \le r$,
	\begin{equation}\label{eq1}
	 (1+o(1))\frac{k-1}{3 \ln 2r}{n\choose r-1}<
	%(1-o(1))\frac{1}{4 \log {r \choose k}}{n \choose r-1}<
	\ex_{\circlearrowright}(n,P_k^r) \le
	\dfrac{(k-1)(r-1)}{r}\dbinom{n}{r-1}.
	\end{equation}
Furthermore, when $k \in \{2,r\}$, the following sharper bounds hold: 
\begin{eqnarray}
\label{eq61}     \ex_{\circlearrowright}(n,P_2^r) &\le&  \frac{1}{2}{n \choose r-1} \\
\label{eq62}     \ex_{\cir}(n, P_r^r) &\geq& (1-o(1))(r-2){n \choose r-1}.
\end{eqnarray}
	\end{thm}

The lower bound in (\ref{eq62}) is close to the upper bound in (\ref{eq1}), since the upper bound is $(r - 2 + 1/r){n \choose r - 1}$. 
We remark that it remains open to prove or disprove that for every $r \geq 2$, there exists $c_r$ such that $c_r \rightarrow 0$ as $r \rightarrow \infty$ and $$\ex_{\circlearrowright}(n, P_2^r) \leq c_r {n \choose r-1} + o(n^{r-1}).$$

Theorems~\ref{cpthm} and~\ref{cgthm} reveal a discrepancy between the ordered setting and the convex geometric setting: in the convex geometric setting, crossing paths of length up to $2r - 1$ have extremal function of order $n^{r - 1}$, whereas this phenomenon only occurs for crossing paths of length up to $r + 1$ in the ordered setting. In fact, we know  that
$\ex_{\circlearrowright}(n, P^r_k)=\ex_{\to}(n, P^r_k)$ iff $k \in \{1, r+1\}$.

\subsection{Crossing matchings}

Let $M_k^2$ denote the cgg consisting of $k$ pairwise crossing line segments. In other words, there is a labelling of the vertices such that the edges of the matching are $v_i v_{k + i}$ for $1 \leq i \leq k$, and $v_1 < v_2 < \dots < v_{2k}$.

Capoyleas and Pach~\cite{Capoyleas-Pach} proved the following theorem which  extended a result of Ruzsa (he proved the case $k=3$) and settled a question of G\"artner and conjecture of Perles:

\begin{thm} [Capoyleas-Pach~\cite{Capoyleas-Pach}] \label{cgg-cmatching}
	For all $n \geq 2k - 1$, $\ex_{\circlearrowright}(n,M_k^2) = 2(k-1)n - {2k - 1 \choose 2}$.
\end{thm}
As mentioned earlier, a related open problem of Bra{\ss}~\cite{Brass} is to determine all acyclic graphs $F$ such that $\ex_{\circlearrowright}(n,F) = O(n)$.
%Examples of small forests with linear extremal function are given in Bra{\ss}~\cite{Brass}, as well as examples with extremal function $\Omega(n\log n)$.

%We generalize Theorem~\ref{cgg-cmatching} as follows.
For $r \ge 2$, an $r$-uniform {\em crossing $k$-matching} $M_k^r$ has vertex set $v_1,v_2,\dots,v_{rk}$ on a convex $n$-gon in clockwise order
and consists of the edges $\{v_i,v_{i+k},\dots,v_{i+(r-1)k}\}$ for $1 \leq i \leq k$.
Note that crossing paths have the property that if we take every $r$th edge of the path, we obtain a crossing matching.

One can similarly define a {\em crossing $k$-matching} $M_k^r$ in ordered $r$-graphs: it has vertex set $v_1,v_2,\cdots,v_{rk}$
with $v_1<v_2<\ldots<v_{rk}$
and consists of the edges $\{v_i,v_{i+k},\dots,v_{i+(r-1)k}\}$ for $1 \leq i \leq k$.
However, if we consider a cg $r$-graph $G_1$ and an ordered $r$-graph $G_2$ with the same set of vertices and the same set of edges
(only  the ordering in $G_1$ is linear and in $G_2$ is circular),
then with our definitions a set $F$ of edges is a crossing matching in $G_1$ if and only if it is a crossing matching in $G_2$.
It follows that $$\ex_{\circlearrowright}(n,M_k^r)=\ex_{\to}(n,M_k^r) \qquad \hbox{ for all $k,r,n$}.$$

Aronov, Dujmovi\v{c}, Morin, Ooms and da Silveira~\cite{Aronov} considered the case $k=2$, $r=3$ and determined the order of magnitude in those cases; our result below provides better bounds.
The $k=2$ case of Theorem \ref{cgh-cmatching} could be viewed as an ordered version of the Erd\H{o}s-Ko-Rado Theorem.

\begin{thm}\label{cgh-cmatching}
	For $n>r>1$,
	$$\ex_{\c}(n, M^r_2) =  {n \choose r} - {n-r \choose r}  $$
	and for fixed $k, r > 2$,
	\[
(1-o(1))	(k-1)r {n \choose r - 1}  \leq \ex_{\circlearrowright}(n,M_k^r) \leq 2(k - 1)(r-1){n \choose r - 1}.
	\]
\end{thm}

Note that, unlike the results on the paths, there are no extra $\log n$ factors in the formulas for crossing matchings.
We were unable to determine the asymptotic behavior of $\ex_{\circlearrowright}(n,M_k^r)$ for any pair $(k,r)$ with $k,r > 2$.

\section{Proof of Theorem~\ref{cpthm}}

\subsection{Upper bound for $k \leq r + 1$}
%First we give the upper bound for $k \leq r + 1$, starting
Observe that $\ex_{\to}(n,P^1_2)=1$ for all $n \ge 1$. We then have the following recurrence:

\begin{prop} \label{firstclaim}
Let $2\leq k\leq r+1$ and $n\geq r+k$. Then
\begin{equation}\label{j12}
\ex_{\rightarrow}(n, P^r_k)\leq {n-2\choose r-2}+\ex_{\rightarrow}(n-2, P^{r-1}_{k-1})+\ex_{\rightarrow}(n-1, P^r_k).
\end{equation}
\end{prop}
\proof Let $G$ be an $n$-vertex ordered $r$-graph  not
containing  $P^r_k$ with $e(G)=\ex_{\rightarrow}(n, P^r_k)$. We may assume $V(G)=[n]$ with the natural ordering.
Let $G_1=\{e\in G: \{1,2\}\subset e\}$ and $G_2=\{e\in G: 1\in e, 2\notin e, e-\{1\}\cup \{2\}\in G\}$.
Let $G_3$ be obtained from $G-E(G_1)-E(G_2)$ by
 gluing vertex $1$ with vertex $2$ into
a new vertex $2'$.

Since we have deleted the edges of $G_1$, our $G_3$ is an $r$-graph,
 and since we have deleted the edges of $G_2$,  $G_3$ has no multiple edges.
Thus $e(G)=e(G_1)+e(G_2)+e(G_3)$.

 We view $G_3$ as an ordered $r$-graph with vertex set $\{2',3,\ldots,n\}$. If $G_3$ contains a crossing ordered path
 $P$ with edges $e'_1,e'_2,\ldots,e'_k$, then only $e'_1$ may contain $2'$, and all other edges are edges of $G$.
 Thus either $P$ itself is in $G$ or the path obtained from $P$ by replacing $e'_1$ with $e'_1-\{2'\}+\{1\}$ or with
$e'_1-\{2'\}+\{2\}$ is in $G$, a contradiction. Thus $G_3$ contains no $P^r_k$ and hence
%\begin{equation}\label{j13}
$$e(G_3)\leq \ex_{\rightarrow}(n-1, P^r_k).$$
%\end{equation}
By definition, $e(G_1)\leq {n-2\choose r-2}$. We can construct an ordered $(r-1)$-graph $H_2$ with vertex set
$\{3,4,\ldots,n\}$ from $G_2$ by deleting from each edge vertex $1$. If $H_2$ contains a crossing ordered path
 $P'$ with edges $e''_1,e''_2,\ldots,e''_{k-1}$, then the set of edges $\{e_1,\ldots,e_k\}$ where $e_1=e''_1+\{1\}$ and
 $e_i=e''_{i-1}+\{2\}$ for $i=2,\ldots,k$ forms a $P^r_k$ in $G$, a contradiction. Summarizing, we get
\begin{eqnarray*}
\ex_{\rightarrow}(n, P^r_k)=e(G) &=& e(G_1)+e(G_2)+e(G_3) \\
&\leq&  {n-2\choose r-2}+\ex_{\rightarrow}(n-2, P^{r-1}_{k-1})+ \ex_{\rightarrow}(n-1, P^r_k),
\end{eqnarray*}
as claimed.\qed

We are now ready to prove the upper bound in Theorem \ref{cpthm} for $k \leq r + 1$: We are to show that
$\ex_{\rightarrow}(n, P^r_k) \leq {n \choose r} - {n - k + 1 \choose r}$.
We use induction on $k+n$. Since $P^r_1$ is simply an edge, $\ex_{\rightarrow}(n, P^r_1)=0$ for any $n$ and $r$, and the theorem holds for $k=1$.

Suppose now the upper bound in the theorem holds for all $(k',n',r')$ with $k'+n'<k+n$ and we want to prove it for $(k,n,r)$. By the previous paragraph,
it is enough to consider the case $k\geq 2$. Then by Proposition \ref{firstclaim} and the induction assumption,
\begin{eqnarray*}\ex_{\rightarrow}(n, P^r_k) &\leq& {n-2\choose r-2}+
\left[{n-2 \choose r-1} - {n-k \choose r-1}\right]+\left[ {n-1 \choose r} - {n-k \choose r}\right] \\
&=&\left[{n-2\choose r-2}  +{n-2 \choose r-1}+{n-1 \choose r}\right] -
\left[  {n-k \choose r}+ {n-k \choose r-1}\right] \\
&=& {n \choose r} - {n-k+1 \choose r},
\end{eqnarray*}
as required. This proves the upper bound in Theorem \ref{cpthm} for $k \leq r + 1$. \qed

\subsection{Lower bound for $k \leq r + 1$}

For the lower bound in Theorem \ref{cpthm} for $k \leq r + 1$, we provide the following construction. For  $1\leq k\leq r$, let $G(n,r,k)$ be the family of $r$-tuples $(a_1,\ldots,a_r)$ of positive integers such that
\begin{center}
\begin{tabular}{lp{5in}}
$(a)$ & $1\leq a_1<a_2<\ldots<a_r\leq n$ and \\
$(b)$ & there is $1\leq i\leq k-1$ such that $a_{i+1}=a_i+1$.
\end{tabular}
\end{center}
Also, let $G(n,r,r+1)=G(n,r,r)\cup \{(a_1,\ldots,a_r): a_1<a_2<\ldots<a_r=n\}$.

\medskip

Suppose $G(n,r,k)$ has an ordered crossing $P_k^r$ with edges $e_1,\ldots,e_k$.
Let $e_1=(a_1,\ldots,a_r)$ where $1\leq a_1<a_2<\ldots<a_r\leq n$.
By the definition of a crossing ordered path, for each $2\leq j\leq k$, $e_j$ has the form
\begin{equation}\label{j14}
\mbox{\em
$e_j=(a_{j,1},\ldots,a_{j,r})$ where $a_i< a_{j,i}<a_{i+1}$ for $1\leq i\leq j-1$ and $a_{j,i}=a_{i}$ for $j\leq i\leq r$.}
\end{equation}
By the definition of $G(n,r,k)$, either there is $1\leq i\leq k-1$ such that $a_{i+1}=a_i+1$ or $k=r+1$ and $a_r=n$.
In the first case, we get a contradiction with~(\ref{j14}) for $j=i+1$. In the second case,
we get a contradiction with~(\ref{j14}) for $j=r+1$.

In order to calculate $|G(n,r,k)|$, consider the following procedure $\Pi(n,r,k)$ of generating all $r$-tuples of elements of $[n]$ {\em not} in
$G(n,r,k)$: take an  $r$-tuple $(a_1,\ldots,a_r)$
of positive integers such that
 $1\leq a_1<a_2<\ldots<a_r\leq n-k+1$ and then increase $a_j$ by $j-1$ if $1\leq j\leq k$ and by $k-1$ if
 $k\leq j\leq r$. By definition, the number of outcomes of this procedure is ${n-k+1\choose r}$. Also  $\Pi(n,r,k)$
 never generates a member of $G(n,r,k)$ and generates each other $r$-subset of $[n]$ exactly once.  \qed

\subsection{Upper bound for $k \geq r + 2$}
An ordered $r$-graph has interval chromatic number $r$ if it is $r$-partite with $r$-partition $A_1, \ldots, A_r$ and  $A_i$ precedes $A_{i+1}$ in the ordering of the vertices for all $i\in [r-1]$.

Let $z_{\rightarrow}(n,F)$ denote the maximum number of edges in an $n$-vertex ordered $r$-graph of interval chromatic number $r$ that does not contain the ordered graph $F$.
Pach and Tardos~\cite{PT} showed that every $n$-vertex ordered graph may be written as the union of at most $\lceil \log n\rceil$ edge disjoint subgraphs each of whose components is a graph of interval chromatic number two, and deduced that $\ex_{\rightarrow}(n,F) = O(z_{\rightarrow}(n,F)\log n)$ for every ordered graph $F$. They also observed that the log factor is not present when $z_{\to}(n,F)=\Omega(n^c)$ and $c>1$. Unsurprisingly, this phenomenon also holds for ordered $r$-graphs when $r>2$. We will 
use the following result which is a rephrasing of~\cite{FJKMV2}, Theorem 1.1.

\begin{thm}[\cite{FJKMV2}, Theorem 1.1]   \label{splitting1}
	Fix $r \ge c\ge r-1\ge 1$ and an ordered $r$-graph $F$ with $z_{\rightarrow}(n, F)=\Omega(n^{c})$. Then
	\[ \ex_{\to}(n, F) = \left\{\begin{array}{ll}
	O(z_{\to}(n, F) \log n) & \mbox{ if } c =r-1 \\
	O(z_{\to}(n, F))& \mbox{ if }c >r-1.
	\end{array}\right.\]
	%$\ex_{\to}(n, F)=O(z_{\to}(n, F) \log n)$  unless $c>r-1$ in which case $\ex_{\to}(n, F)=O(z_{\to}(n, F))$.
\end{thm}

By Theorem~\ref{splitting1}, the following claim yields $\ex_{\rightarrow}(n, P_k^r) = O(n^{r-1}\log n)$ for all $k \ge 2$, i.e., the upper bound in  Theorem~\ref{cpthm} for $k \geq r + 2$.

\begin{prop} \label{secondclaim}
For $k \geq 1$, $r \geq 2$, $z_{\rightarrow}(n,P_k^r) = O(n^{r - 1})$.
\end{prop}

%In what follows, we suppose that the underlying set (the set of vertices) of an ordered hypergraph is $[n]$.
%An {\em interval} is a set of consecutive vertices in the ordering.

{\bf Proof.} %of Proposition~\ref{secondclaim}.}
  We prove a stronger statement by induction on $k$: {\em if $H$ is an ordered $n$-vertex $r$-graph of interval chromatic number $r$ with $r$-partition $X_1,X_2,\dots,X_r$ of sizes $n_1,n_2,\dots,n_r$ respectively, and $H$ has no crossing $k$-path, then
$e(H) \le k P$ where
$$P=\prod_{i = 1}^r n_i \cdot \sum_{i=1}^r \frac{1}{n_i}.$$}
The base case $k=1$ is trivial. For the induction step, assume the result holds for paths of length at most $k-1$, and suppose $e(H) > kP$. For each $(r-1)$-set $S$ of vertices mark the edge $S \cup \{w\}$ where $w$ is maximum.  Let $H'$ be the $r$-graph of unmarked edges. Since we marked at most $P$ edges, $e(H') > (k-1)P$. By the  induction
assumption there exists a $P^r_{k-1} =v_1 v_2 \ldots v_{k+r-2}  \subset H'$ and we can extend this to a $P^r_k$ in $H$ using the marked edge obtained from the $(r-1)$-set $\{v_{k}, \ldots, v_{k+r-2}\}$. This proves the proposition. \qed

\subsection{Lower bound for $k \geq r + 2$}\label{orderedconstruction}

We now turn to the lower bound in Theorem~\ref{cpthm}.
Let $G(n,r,r+2)$ be the family of $r$-tuples $(a_1,\ldots,a_r)$
of positive integers such that
\begin{center}
\begin{tabular}{lp{5in}}
$(a)$ & $1\leq a_1<a_2<\cdots<a_r\leq n$ and \\
$(b)$ & $a_{2}-a_1=2^p$, where $p\leq \log_2 (n/4)$ is an integer.
\end{tabular}
\end{center}
The number of choices of $a_1\leq n/4$ is $n/4$,  then the number of choices of $a_2$ is $\log_2 (n/4)$, and
the number of choices of the remaining $(r-2)$-tuple $(a_3,\ldots,a_r)$ is at least ${n/2\choose r-2}$.
Thus if $r\geq 3$ and $n>20r$, then
\begin{equation}\label{j15}
|G(n,r,r+2)|\geq \frac{n^{r-1}}{(r-2)!3^{r}}\log_2 n.
\end{equation}

Suppose $G(n,r,r+2)$ contains a $P_{r+2}^r$ with  vertex set  $\{a_1, \ldots, a_{2r+1}\}$ and edge set
 $\{a_i\ldots a_{i+r-1}: 1 \le i \le r+2\}$. By the definition of ordered path, the vertices are in the following order on $[n]$:
  \begin{equation}\label{j16}
  a_1<a_{r+1}<a_{2r+1}<a_2<a_{r+2}<a_3<a_{r+3}<\ldots<a_r<a_{2r}.
 \end{equation}
 Hence the 2nd, $r+1$st and $r+2$nd edges are
$$\{a_{r+1},a_2,a_3 \ldots, a_{r}\}, \qquad \{a_{r+1}, a_{r+2}\ldots, a_{2r}\}, \qquad \{a_{2r+1},a_{r+2}, \ldots, a_{2r}\}.
$$
 The differences between the second and the first coordinates  in these three vectors are
$$d_1=a_{2}-a_{r+1} , \qquad  d_2= a_{r+2}-a_{r+1}, \qquad d_3= a_{r+2}-a_{2r+1}.$$
By~(\ref{j16}), we have $d_1,d_3<d_2<d_1+d_3$ so it is impossible that all the three differences $d_1, d_2, d_3$ are  powers of two.
This yields the lower bound in Theorem~\ref{cpthm} for $k\geq r+2$. \qed

\section{Proof of Theorem \ref{cgthm}}
We begin with the  upper bounds when $r+1< k \le 2r-1$.

\begin{definition}
	An ordered $r$-graph $F$ is a {\em split hypergraph} if there is a partition of $V(F)$ into intervals $X_1<X_2<\dots<X_{r - 1}$ and there exists $i \in [r-1]$ such that  every edge of $F$ has two vertices in $X_i$ and one vertex in every $X_j$ for $j \ne i$.
\end{definition}

Every $r$-graph of interval chromatic number $r$ is a split hypergraph (but not vice versa). We write $e(H)$ for the number of edges in a hypergraph $H$, $v(H) = \bigl|\bigcup_{e \in H} e\bigr|$ and $d(H) = e(H)/v(H)^{r - 1}$.
The function $d(H)$ could be viewed as a normalized average degree of $H$. We require the following nontrivial result about split hypergraphs.

\begin{thm} [\cite{FJKMV2}, Theorem 1.2] \label{splitter}
	For $r\geq 3$  there exists $c=c_r>0$ such that every ordered $r$-graph $H$ contains a split subgraph $G$  with $d(G) \geq c\, d(H)$.
\end{thm}

\begin{prop} \label{prop}
For $r\ge 3$ there exists $C=C_r>0$ such that, if $r+1<k\le 2r-1$, then	$\ex_{\cir}(n, P_k^r)\le k C\, n^{r-1}$.
	\end{prop}

\proof
Let $c=c_r$ be the constant from Theorem~\ref{splitter} and let $C=1/c$. Given a convex geometric $r$-graph $H$ with $e(H) > k \,C n^{r-1}$, we view $H$ as a linearly ordered $r$-graph (by ``opening up" the circular ordering between any two vertices) and apply Theorem~\ref{splitter}  to obtain a split subgraph $G \subset H$
where $e(G) > k m^{r-1}$ where $m=v(G)$. Now, viewing $H$ once again as a convex geometric $r$-graph, let $X_0 < X_1 < \dots < X_{r-3} < X$ be cyclic intervals such that every edge of $G$ contains two vertices in
$X$ and one vertex in each $X_i : 0 \leq i \leq r - 3$. Our main assertion is the following:

	 For $k \in [2r-1]$, $G$ contains a crossing $k$-path $v_0 v_1 \ldots v_{k+r-2}$ such that
	
$\bullet$  $v_i \in X_i$ for $i\not \equiv -1, -2 \mod r$ and

	$\bullet$  $v_i \in X$ for $i \equiv -1, -2 \mod r$.

To prove this assertion  we proceed by induction on $k$, where the base case $k = 1$ is trivial. For the induction step, suppose that $1\le k \le 2r-2$, and we have proved the result for $k$ and we wish to prove it for $k+1$. Suppose that $k \equiv i \not\equiv 0, -1$ (mod $r$) where $0 \le i<r$. For each $f \in \partial G$ that has no vertex in $X_{i-1}$,  delete the edge $f \cup v \in G$
where $v$ is the largest vertex in $X_{i-1}$ in clockwise order. Let $G'$ be the subgraph that remains after deleting these edges. Then
$$e(G')\ge e(G)-m^{r-1}>(k+1)m^{r-1}-m^{r-1}=km^{r-1},$$ so by induction $G'$ contains  a $P_{k}^r$  with vertices $v_0, v_1, \ldots, v_{k-1},\ldots, v_{k+r-2}$, where $v_i \in X_i$ for $i\not \equiv -1, -2$ (mod $r$) and $v_i \in X$ for $i \equiv -1, -2$ (mod $r$). Our goal is to add a new vertex $v$ to the end of the path where $v \in X_{i-1}$. Let $v=v_{k+r-1}$ be the vertex in $X_{i-1}$ for which the edge $e_k=v_k v_{k+1} \ldots v_{k+r-1}$ was deleted in forming $G'$. Note that $v$ exists as $v_{k-1} v_k \ldots v_{k+r-2} \in E(G)$ and so $v_k \ldots v_{k+r-2} \in \partial G.$
Adding vertex $v$ and edge  $e_k$ to our copy of $P_k^r$ yields a copy of $P_{k+1}^r$ as required.

Next suppose that $i \equiv 0,-1$ (mod $r$). Proceed exactly as before except we modify the definition of $G'$ slightly as follows: for every $f \in \partial G$ which has exactly one vertex in each $X_i$ and in $X$, if $w$ is the vertex of $f$ in $X$, then delete $f \cup v \in G$ where $v$ is the largest such vertex in $X$ satisfying $v<w$.

By induction,  $G'$ contains a $P_{k}^r$   with vertices $v_0, v_1, \ldots, v_{k-1},\ldots, v_{k+r-2}$, where $v_i \in X_i$ for $i\not \equiv -1, -2$ (mod $r$) and $v_i \in X$ for $i \equiv -1, -2$ (mod $r$). Our goal is to add a new vertex $v$ to the end of the path where $v \in X$ so we may assume that $k\in \{r-1, r\}$, and we are trying to find vertex $v$ which we will label as $v_{k+r-1} \in \{v_{2r-2}, v_{2r-1}\}$ as above with  $v \in X$. Note that we already have the two vertices $v_{r-2} < v_{r-1}$ in $X$. So we either want to add $v_{2r-2}$ satisfying
 $v_{r-2} < v_{2r-2}< v_{r-1}$ or we want to add $v_{2r-1}$ satisfying
 $v_{r-2} < v_{2r-2}< v_{r-1} < v_{2r-1}$.
 Suppose that $k=r-1$ so that we are in the first case.
 Since $v_{r-2} \ldots v_{2r-3} \in E(G')$, the $(r-1)$-set $f =v_{r-1} \ldots v_{2r-3}$  has exactly one vertex  $v_{r-1} \in X$.  Since $f \cup \{v_{r-2}\}  = v_{r-2}v_{r-1} \ldots v_{2r-3} \in E(G')$, we have $f \in \partial G$ and moreover $v_{r-2}$ was not deleted from $f \cup \{v_{r-2}\}$ if forming $G'$.  Hence there is a vertex $v \in X$ with $v_{r-2}<v<v_{r-1}$ such that the edge $f \cup \{v\}  = v_{r-1} \ldots v_{2r-3}v \in E(G)$
 and the vertex $v$ and edge  $f \cup \{v\}$ can be used to extend the $P_k^r$ to a $P_{k+1}^r$. For the case $k=r$, we choose $v$ to be the largest vertex in $X$ in defining $G'$ and apply an identical argument to that when $i\not \equiv -1, -2$ (mod $r$) . \qed

Next we give lower bounds for $k \geq 2r$.

\begin{prop} \label{prop2}
	For $k \ge 2r\ge 4$ we have $\ex_{\cir}(n, P_k^r) =\Omega(n^{r-1} \log n).$
	\end{prop}
We take the same family $G(n,r,r+2)$ as used for ordered hypergraphs (see Section \ref{orderedconstruction}),
but with the cyclic ordering of the vertex set.  When we have a $k$-edge crossing path $P=w_1w_2\ldots w_{r+k-1}$, the vertex $w_1$ does not need to be the leftmost in the first
edge $w_1\ldots w_r$, so the argument in Section~\ref{orderedconstruction} does not go through for $k=r+2$. In fact,   $G(n,r,r+2)$ does contain
$P_k^r$ for $k \leq 2r-1$.

However,
%\begin{equation}\label{Ne} \parbox{14.5cm}{\em it does not have a crossing $(r+2)$-edge path $a_1a_2\ldots a_{2r+1}$
%in which \\  $a_{r+1}<a_{2r+1}<a_2<a_{r+2}<a_r$, } \end{equation}
%since otherwise we can repeat the argument of Section~\ref{orderedconstruction}.
% We now show  that $G(n,r,r+2)$ does not contain $CP_{2r}^r$. S
 suppose $G(n,r,r+2)$ has  a crossing $2r$-edge path $P=w_1\ldots w_{3r-1}$, and the $i$th edge of the path is $A_i=w_iw_{i+1}\ldots w_{i+r-1}$.
 Suppose vertex $w_{r+j}$ is the leftmost
in the set $\{w_{r},w_{r+1},\ldots,w_{2r-1}\}$. % and $w_1$ is the $i$th smallest in the first edge, then
%$w_2$ is the $(i+1)$st smallest in the second edge and so on (modulo $r$).
%If $j\geq 1$, then the path $a_1a_2\ldots a_{2r+1}$ where $a_i=w_{i+j-1}$ for $i=1,\ldots,2r+1$ sastisfies~(\ref{Ne}), a contradiction. If $j=0$, then
 Then writing the edges $A_{j+1}, A_{j+r}$ and $A_{j+r+1}$ as vectors with increasing coordinates, we have %the first, $r$th and $r+1$st edges are
$$A_{j+1}=\{w_{j+r}, w_{j+1},w_{j+2}, \ldots, w_{j+r-1}\}, \quad A_{j+r}= \{w_{j+r}, w_{j+r+1}\ldots, w_{j+2r-1}\}, $$
$$\mbox{and }\quad
A_{j+r+1}=\{w_{j+2r},w_{j+r+1},w_{j+r+2}, \ldots, w_{j+2r-1}\}.
$$
 The differences between the second and the first coordinates  in these three vectors are
$$d_1=w_{j+1}-w_{j+r} , \qquad  d_2= w_{j+r+1}-w_{j+r}, \qquad d_3= w_{j+r+1}-w_{j+2r}.$$
As at the end of Section~\ref{orderedconstruction},
it is impossible that all the  differences $d_1, d_2, d_3$ are  powers of two.
 \qed
\medskip

{\bf Proof of Theorem~\ref{cgthm}.}
Proposition~\ref{prop} yields $C=C_r$ such that
\[ \ex_{\cir}(n,P_k^r) \leq kC\, n^{r-1}\]
for $k \le 2r-1$.
Since the family of all $r$-subsets of $[n]$ containing $1$ witnesses that
 for $k\geq 3$, $r \geq 2$, $\ex(n,P_k^r) = \Omega(n^{r - 1})$, and
$\ex_{\cir}(n,P_k^r) \geq \ex(n,P_k^r)$, we get $\ex_{\cir}(n,P_k^r) = \Theta(n^{r - 1})$ for $3\leq k \leq 2r - 1$. In the case $k = r + 1$, Theorem~\ref{cpthm} gives $$\ex_{\cir}(n,P_{r+1}^r) \leq \ex_{\rightarrow}(n,P_{r+1}^r) = {n \choose r} - {n - r \choose r}.$$ On the other hand, since $P_{r+1}^r\supseteq M_r^2$ and    $G(n, r, r+1) \not\supseteq M_r^2$,
$$\ex_{\cir}(n,P_{r+1}^r) \geq \ex_{\cir}(n,M_r^2) =\ex_{\to}(n, M_r^2) \ge |G(n, r, r+1)| = {n \choose r} - {n - r \choose r},$$ so the second statement in Theorem \ref{cgthm} follows.
It remains to consider  $k \geq 2r$, and here we have $$\ex_{\cir}(n,P_k^r) \leq \ex_{\rightarrow}(n,P_k^r) = O(n^{r - 1}\log n)$$ from Theorem \ref{cpthm} and a  lower bound from Proposition~\ref{prop2}. \qed

\section{Proof of Theorem~\ref{maincrossing_short}}

\subsection{Upper bound in Theorem~\ref{maincrossing_short} for $r \ge k \ge 2$}

Let us first prove the upper bound
\begin{equation}\label{j5}
%$$  % (1-o(1))\frac{k-1}{16 \log r}{n \choose r-1}<
\ex_{\circlearrowright}(n,P_k^r) \le
\frac{(k-1)(r-1)}{r}{n \choose r-1} \qquad \hbox{($2 \le k \le r$)}.
\end{equation}

Recall that our notation for a crossing $k$-path $P_k^r$ ($k \le r$) on a cyclically ordered vertex set $\vb{\Omega}_n$ is the following: the vertices $v_1, v_2, \ldots, v_{r+k-1}$ form a tight path with edges $e_i=\{v_i, \ldots, v_{i+r-1}\}$, $i \in [k]$ and the (clockwise) ordering of the vertices on
$\vb{\Omega}_n$ is
$$v_1<v_{r+1}< v_2<v_{r+2}<\cdots <v_{k-1}<v_{r+k-1}<v_k<v_{k+1} <\cdots < v_r \enskip (< v_1).$$
We define $T_k(H)$ to be the set of $(v_{k}, \ldots, v_{r+k-1}) \in V(H)^{r}$ for which
there is a $P_k^r$ in $H$ with vertices $v_1, \ldots, v_{r+k-1}$ as ordered above.
In other words, $T_k(H)$ is the set of ending edges for a $P_k^r$ in $H$.

\begin{thm}\label{zigzag}
	Let $r \ge 2$ and $1\le k\le r$. Then for any  cg $r$-graph $H$ on $\vb{\Omega}_n$,
	\[ |T_k(H)| \geq r \cdot e(H) - (r - 1)(k - 1)\cdot |\partial H|.\]
	In particular, if $H$ contains no $P_k^r$, then
	\[ e(H) \leq \frac{(k-1)(r-1)}{r}|\partial H| \leq \frac{(k - 1)(r - 1)}{r}{n \choose r - 1}.\]
\end{thm}

\begin{proof} We proceed by induction on $k$. For $k = 1$, and each edge $e \in E(H)$, the number of copies of $P_1^r$ with edge set $\{e\}$ is $r$,
	since after choosing which vertex of $e$ to label with $v_1$, the order of the remaining vertices of $e$ is determined (they are cyclically ordered). Therefore $|T_1(H)| \geq re(H)$. Suppose $k \geq 2$ and assume by induction that $|T_{k-1}(H)| \geq r e(H) - (r - 1)(k - 2)|\partial H|$.
	Let $L$ be  the collection of $r$-sets in $T_{k-1}(H)$ with the following property: The elements of $L$ are
	$$e= x_{r+1}< \cdots < x_{r+k-1}<x_k<\cdots <x_r$$ where $e \in E(H)$ and there does not exist any vertex $x$ such that $x_k<x<x_{k+1}$ and
	$e-\{x_k\} \cup \{x\} \in E(H)$. Observe that $|L|\le (r-1)|\partial H|$ since for each  ordered $(r-1)$ set $e-\{x_k\} \in \partial H$ there must be a unique $x_k$ satisfying $x_{r+k-1}<x_k<x_{k+1}$ such that $e \in L$ (the vertex closest to $x_{k+1}$).
	Our goal is to prove that
	$|T_k(H)| \geq |T_{k-1}(H) \backslash L|$ via an injection. Then, using the fact that $|L|\le (r-1)|\partial H|$ and the induction hypothesis, we have
	\[ |T_k(H)| \geq |T_{k-1}(H) \backslash L| \geq r \cdot e(H) - (k - 2)(r - 1) \cdot |\partial H| - |L| \geq r \cdot e(H) - (k - 1)(r - 1) \cdot |\partial H|.\]
	
	We must give an injection $f : T_{k-1}(H) \backslash L \rightarrow
	T_k(H)$. Suppose that $e= v_{r+1}< \cdots < v_{r+k-1}<v_k<\cdots <v_r \in T_{k-1}(H) \backslash L$.
	Then  there exists a vertex $x$ such that  $v_k<x<v_{k+1}$ and
	$e-\{v_k\} \cup \{x\} \in E(H)$. Let $A$ be the set of all such vertices $x$.
	Consider the vertex $y \in A$ such that $y \le x$ for all $x \in A$. In other words, $y$ is the closest vertex to $v_k$ among all vertices of $A$. Let $f(e)=e-\{v_k\} \cup \{y\}$. Since $k \le r$, we clearly have $f(e) \in T_k(H)$ as we obtain a $P_k^r$ that ends in $f(e)$ by taking the copy of $P_{k-1}^r$ that ends in $e$ and just adding the edge $f(e)$. Moreover, $f$ is an injection, as if
	there is an $e'=e-\{v_k\} \cup \{y'\}$ such that $f(e')=f(e)$, then, assuming %%% wlog
	that $v_k<y'<y$, $y$ would not have been the closest vertex to $v_k$ in $A$. This contradiction shows that $f$ is indeed an injection and the proof is complete.
\end{proof}

\subsection{Lower bound in Theorem~\ref{maincrossing_short} for $r \ge k \ge 2$}

Our next goal is to prove the following lower bound in Theorem~\ref{maincrossing_short} for $r \ge k \ge 2$:

\begin{equation} \label{logr}
 \ex_{\circlearrowright}(n,P_k^r) \geq (1+o(1))\frac{k-1}{3 \ln 2r}{n\choose r-1}. % \qquad \hbox{($r \ge k \ge 2$. %, $n \rightarrow \infty$
\end{equation}

A {\em gap} of an $r$-element subset $R$ of $\vb{\Omega}_n$ is a segment of $\vb{\Omega}_n$ between two clockwise
consecutive vertices of $R$. We say $R$ has {\em $(k,m)$-gaps} if some $k - 1$ consecutive gaps of $R$ all have length
more than $m$ -- in other words, there are at least $m$ vertices of $\vb{\Omega}_n$ in each gap.
For $n>r$, let $K_n^r$  be the family of all $r$-element subsets of $\vb{\Omega}_n$.
For $n>r\geq k$, let  $H(n,r,k,m)$ be the family of the members of $K_n^r$
 that have $(k,m)$-gaps, and  $\overline{H}(n,r,k,m)$ be the family of the members of $K_n^r$
 that do not have $(k,m)$-gaps.

 For a hypergraph $H$ and $v\in V(H)$, let $H\{v\}$ denote the set of edges of $H$ containing $v$.

\begin{lem} If
	\begin{equation}\label{jf1}
m\geq \frac{(n-1)\ln 2r}{(r-1)(k-1)},
 \end{equation}
then	
\begin{equation}\label{jf2}
|H(n,r,k,m)|\leq \frac{1}{2} {n\choose r}. \quad\mbox{ Equivalently,}\; |\overline{H}(n,r,k,m)|\geq \frac{1}{2} {n\choose r}.
 \end{equation}
	\end{lem}

\begin{proof} Instead of proving~(\ref{jf2}) directly, it will be easier to prove that
\begin{equation}\label{jf3}
\mbox{\em for every  $j\in \vb{\Omega}_n$,}\quad |H(n,r,k,m)\{j\}|\leq \frac{1}{2} |K^r_n\{j\}|=  \frac{1}{2}{n-1\choose r-1};
 \end{equation}
and~(\ref{jf3}) implies~(\ref{jf2}) because  $|H(n,r,k,m)|=\frac{n}{r}|H(n,r,k,m)\{j\}|$ and ${n\choose r}=\frac{n}{r} |K^r_n\{j\}|$.

Recall the vertex set of $\vb{\Omega}$ is $\{0,1,2,\dots,n-1\}$.
By symmetry, it is enough to prove~(\ref{jf3}) for $j=n-1$. First, we show that
\begin{equation}\label{jf4}
|H(n,r,k,m)\{n-1\}|\leq r  |K^r_{n-(k-1)m}\{n-1-(k-1)m\}|.
 \end{equation}
Indeed, from each $F\in H(n,r,k,m)\{n-1\}$, we can get an $F'\in K^r_{n-(k-1)m}\{n-1-(k-1)m\}$ by
deleting the first $m$ vertices in $k-1$ consecutive gaps of length at least $m+1$, and renumbering the remaining $n-(k-1)m$ vertices so that the vertex $n-1$  of $\vb{\Omega}$ will be
$(n-1)-(k-1)m$. On the other hand, each $F'\in K^r_{n-(k-1)m}\{n-1-(k-1)m\}$ can be obtained this way from $r$ distinct $F\in H(n,r,k,m)\{n-1\}$.
This proves~(\ref{jf4}).

Now, using $1 - x \leq e^{-x}$,~(\ref{jf4}) and~(\ref{jf1})  yield
$$|H(n,r,k,m)\{n-1\}|\leq r {n-1-(k-1)m\choose r-1}= r {n-1\choose r-1} \prod_{i = 1}^{r-1} \frac{n - (k - 1)m - i}{n - i}
$$
$$
\leq r {n-1\choose r-1}\exp\Bigl(-\frac{(k - 1)m(r-1)}{n -1 }\Bigr) \leq  r {n-1\choose r-1}\frac{1}{2r},
$$
yielding~(\ref{jf3}).
\end{proof}

We are ready to prove~(\ref{logr}).     Let
$$t=t(r,k)=\left\lceil \frac{(r-1)(k-1)}{\ln 2r}\right\rceil.$$

Suppose  $n>r\geq k\geq 2$. If $r=2$, then $k=2$, and the bound is trivial; so let $r\geq 3$. Suppose first that $t$ divides $n$ and let
$m=n/t$. Then
 $m$ satisfies~(\ref{jf1}).
By rotating $\vb{\Omega}$ we find a subgraph $H'$ of $\overline{H}(n,r,k,m)$  with at least $|\overline{H}(n,r,k,m)|/m$ edges
such that  every edge of $H'$ adds up to zero modulo $m$. We claim that
\begin{equation}\label{jf5}
\mbox{\em $H'$ does not contain crossing $P_k^r$.}
 \end{equation}
Indeed, assume $H'$ contains a  crossing $P_k^r$  with the vertices  $v_0,v_1,\dots,v_{k+r-2}$. By the definition
of crossing paths, $v_0 < v_r < v_1 < v_{1+r} < \dots < v_{k-1} < v_{k -1+ r} < v_k$.  Since the set $\{v_1,v_2,\dots,v_{r-1}\}$ forms an edge together
with both $v_0$ and $v_r$,  $v_r \equiv v_0 \mod m$. Similarly, $v_{r + i} \equiv v_i \mod m$ for all $i < k$. But this means that  the edge
$\{v_0,v_1,\dots,v_{r-1}\}$ has
$k - 1$ consecutive gaps of length
more than $m$, thus it does not belong to $\overline{H}(n,r,k,m)$. This contradiction proves~(\ref{jf5}).

\medskip
Thus if $r\geq 3$, $2\leq k\leq r$ are fixed, $n$ is a large number divisible by $t$ and $m=n/t$, then
by~(\ref{jf5}) and~(\ref{jf2}), $H'$ is a cg $r$-graph not containing crossing $P_k^r$ with
$$|H'|\geq \frac{1}{2m}{n\choose r}\geq \frac{t}{2r }{n-1\choose r-1}\geq  \frac{(k-1)(r-1)}{2r \ln 2r}{n-1\choose r-1}
\geq  (1+o(1))\frac{k-1}{3 \ln 2r}{n\choose r-1}
.
$$
If $n$ is not divisible by $t$, then let $n'$ be the largest positive integer divisible by $t$ such that $n'\leq n$. Then
$$ \ex_{\circlearrowright}(n,P_k^r) \geq \ex_{\circlearrowright}(n',P_k^r) \geq  (1+o(1))\frac{k-1}{3 \ln 2r}{n'\choose r-1}=
 (1+o(1))\frac{k-1}{3 \ln 2r}{n\choose r-1}.\qed
$$

\subsection{The case $k = 2$} 

Here we prove the upper bound (\ref{eq61}), namely:
\begin{equation*}
    \ex_{\circlearrowright}(n,P_2^r) \le  \frac{1}{2}{n \choose r-1}.  
\end{equation*}
Recall that $P_2^r$ on $\vb{\Omega}_n$ has a vertex set
$$v_1<v_{r+1}< v_2<v_{3}<\cdots < v_r \enskip (< v_1), $$
and edges $\{ v_1, \dots, , v_r\}$ and $\{ v_2, \dots, v_{r+1}\}$.
Consider a $P_2^r$-free cgh $H$ on the vertex set  $\vb{\Omega}_n$.
Label the vertices of an $e\in H$ as
$$1\leq a_1< a_2<\cdots < a_r  \leq n, $$
and define $T_1(e):= e\setminus \{ a_1\}$ and $T_2(e):= e\setminus \{ a_r\}$.
Since $H$ is  $P_2^r$-free, we have $T_\alpha(e)\neq T_\alpha(e')$ for $e\neq e'\in H$ (and $\alpha=1,2$).
Indeed, if we  take (in case of $\alpha=1$) $v_2, \dots, v_r= a_2, \dots, a_r$ and $\{ v_1, v_{r+1}\}= \{ a_1, a_1'\}$ then we obtain a $P_2^r$.

We also have $T_1(e)\neq T_2(e')$, otherwise we define  $\{ v_1, v_{r+1}\}= \{ a_1, a_r'\}$ and again obtain a forbidden path.
This way we associated two $(r-1)$-sets to each member of $H$, yielding~\eqref{eq61}. \qed

\subsection{The case $k = r$} 

Here we prove (\ref{eq62}), namely:
\begin{equation*}
  \ex_{\cir}(n, P_r^r) >  (1-o(1))(r-2){n \choose r-1}.
   \end{equation*}
Recall that $P_r^r$ on $\vb{\Omega}_n$ has a vertex set
\begin{equation}\label{eq63}
 v_1<v_{r+1}< v_2<v_{r+2}<  v_{3}<\cdots < v_{r-1}< v_{2r-1}< v_r \enskip (< v_1),
   \end{equation}
and edges $e_1,\ldots,e_r$, where for $i=1,\ldots,r$, 
  $e_i=\{ v_i,    v_{i+1},\dots    , v_{r+i-1}\}$. 
  %$e_2=\{ v_{r+1},    v_2,\dots    , v_{r-1}, v_r\}$,
 % $e_3=\{ v_{r+1},    v_{r+2},\dots    , v_{r-1}, v_r\}$,  \dots,
%  $e_{r}=\{ v_{r+1},    v_{r+2},\dots    , v_{2r-1}, v_r\}$.
By~(\ref{eq63}),
\begin{equation}\label{m27}
\mbox{\em for every $1\leq i\leq r$, the only vertices in $e_i$ that can be consecutive on $\vb{\Omega}_n$ are  $v_{i+r-1}$ and $v_i$.}
   \end{equation}

Recall that  the $n$ vertices of $\vb{\Omega}_n$ are arranged in clockwise order as $1<2<3< \dots<n$.
Let  $H$ be the following family of $r$-sets  of $\vb{\Omega}_n$.
Label the vertices of an $e\in H$ as
\begin{equation}\label{m28}
1 < a_1< a_2<\cdots < a_r  < n, 
  \end{equation}
and put $e$ into $H$ if %$a_{r-1}+1 < a_r$ and 
 there exists $2\leq i\leq r-1$ with $a_{i-1}+1=a_{i}$.
The number of such $e\in H$ is asymptotically  $(r-2){n \choose r-1}+O(n^{r-2})$.

\medskip

We claim that $H$ does not contain a $P_r^r$.
Suppose, on the contrary, that $F\subset H$ is a copy of $P_r^r$ as it is described in~\eqref{eq63}.
Choose  $i\in [r-1]$ such that the largest  number  in $\{v_1,\ldots,v_{2r-1}\}$ is either $v_i$ or $v_{r+i-1}$.
Consider $e_i$ in the form $(a_1,\ldots,a_r)$ as in~(\ref{m28}). 
%If $1 \in (v_i,v_{r+i}]$ for some $i : 1\le i \le r-1$, then $e_{i}$ cannot be in $H$ as it has no consecutive vertices 
%in $\vb{\Omega}_n$ except the last two. If $1 \in (v_{r+i-1}, v_i]$ for some $i : 1\le i\le r-1$, then again $e_{i}$ is not in $H$, since $e_i$ has no consecutive vertices in $\vb{\Omega}_n$.  
 Since $e_i=\{ v_i,    v_{i+1},\dots    , v_{r+i-1}\}$, by the choice of $i$,  $v_{i+r-1}\in \{a_{r-1},a_{r}\}$.
This together with~(\ref{m27}) contradicts the definition of $H$.
\qed

\section{Proof of Theorem \ref{cgh-cmatching}}\label{pf:cmatching}

%Unlike for crossing paths, we do not need a factor of $\log n$ for large crossing matchings.
We are to show that for  $k,r > 2$,
\[
(k-1)r {n \choose r - 1} - O(n^{r - 2}) \leq \ex_{\circlearrowright}(n,M_k^r) = \ex_{\to}(n,M_k^r)< 2(k - 1)(r-1){n \choose r - 1}.
\]

A simple construction demonstrating the lower bound in Theorem \ref{cgh-cmatching} is the following cgh :
%that generalizes the construction that we used above for $k=2$:
let $A$ be the set of $r$-gons  that contain at least one vertex from a fixed set of $k-1$ vertices of a convex $n$-gon, and
let $B$ be the set of $r$-gons that have a side of length at most $k-1$. The cgh $A \cup B$ has $(k - 1)r{n \choose r - 1} + O(n^{r - 2})$ edges
and does not contain $M_k^r$.

For the upper bound, let $H$ be a largest $r$-uniform $n$-vertex family of sets with vertices on a convex polygon  of $n$ points with no $M_{k}^r$. For each edge $A$, choose a shortest chord $ch(A)$, say $v_rv_1$ and view the vertices of $A$ as $v_1,v_2,\ldots,v_r$ in clockwise order. Define
the {\em type of $A$} to be  the vector ${\bf t}(A)=(t_1,\ldots,t_{r-1})$ where
$$\mbox{$t_i=v_{i+1}-v_i$ for $i=1,\ldots,r-2$ and
	$t_{r-1}=n-(t_1+\ldots+t_{r-2})=v_1-v_{r-1}$.}$$
%,v_4-v_3,\ldots,v_r-v_{r-1},v_2-v_r)$
%, where $v_{i+1}-v_i$ denotes the distance from $v_i$ to $v_{i+1}$ on
%$C$ in clockwise order.
The coordinates of each vector ${\bf t}(A)$ are positive integers, $t_{r-1}(A)\geq 2$, and $t_1(A)+\ldots+t_{r-1}(A)=n$
 for each $A$ by definition.
The number of such vectors is {\em exactly} $\binom{n-2}{r-2}$ (because this is the number of ways to mark $r-2$
out of the $n-1$ separators in an ordered set
of $n$ dots so that the last separator is not marked).  For every given type ${\bf t}=(t_1,\ldots,t_{r-1})$, the family $H({\bf t})$ of the chords $ch(A)$ of the edges $A$ of type ${\bf t}$ does not contain
$k$ crossing chords. Thus by Theorem \ref{cgg-cmatching}, $|H({\bf t})|< 2(k-1)n$.
Hence, using $r\geq 3$,
$$|H|<2(k-1)n\binom{n-2}{r-2}= 2(k-1)\frac{(r-1)(n-r+1)}{n-1}\binom{n}{r-1}<2(k - 1)(r-1){n \choose r - 1},
$$
as claimed. \qed

\section{Concluding remarks}

$\bullet$ A hypergraph $F$ is a {\em forest} if there is an ordering of the edges $e_1,e_2,\dots,e_t$ of $F$ such that for all $i \in \{2,3,\dots,t\}$, there exists $h < i$ such that $e_i \cap \bigcup_{j < i} e_j \subseteq e_h$. It is not hard to show that $\ex(n,F) = O(n^{r - 1})$
 for each $r$-uniform forest $F$. It is therefore natural to extend the Pach-Tardos Conjecture~\ref{ptc} to $r$-graphs as follows:

\begin{conjecture}\label{main}
Let $r \geq 2$. Then for any ordered $r$-uniform forest $F$ with interval chromatic number $r$, $\ex_{\rightarrow}(n,F) = O(n^{r-1} \cdot \mbox{\rm polylog} \, n)$.
\end{conjecture}

Theorem \ref{splitting1} shows that to prove Conjecture \ref{main}, it is enough to consider the setting of $r$-graphs of interval chromatic number $r$. Theorem \ref{cpthm}
verifies this conjecture for crossing paths, and also shows that the $\log n$ factor in Theorem \ref{splitting1} is necessary. It would be interesting to find other general classes of ordered $r$-uniform forests for $r \geq 3$ for which Conjecture \ref{main} can be proved. A related problem is to determine for which ordered forests $F$ we have $\ex_{\to}(n, F)= O(n^{r-1})$? This is a hypergraph generalization of Bra\ss' question~\cite{Brass} which was solved recently for trees~\cite{FKMV}.
\medskip

$\bullet$ It appears to be substantially more difficult to determine the exact value of the extremal function for $r$-uniform crossing $k$-paths in the convex geometric setting than in the ordered setting. It is possible to show that for $k \leq 2r - 1$,
\[ c(k,r) = \lim_{n \rightarrow \infty} \frac{\ex_{\cir}(n,P_k^r)}{{n \choose r-1}}\]
exists. We do not as yet know the value of $c(k,r)$ for any pair $(k,r)$ with $2 \leq k \leq r$, even though in the ordered
setting Theorem \ref{cpthm} captures the exact value of the extremal function for all $k\leq r+1$, and $c(r+1,r) = r$.

\medskip

$\bullet$ One can consider more general orderings of tight paths, namely instead of the vertices whose subscripts are congruent to $a$ modulo $r$ increasing within an interval (conditions (i), (ii), (iii) in Definition~\ref{defCP}), we can specify which congruence classes of vertices are increasing  within their interval and which are decreasing. Our methods can handle such situations as well.

\paragraph{Acknowledgement.}

This research was partly conducted  during AIM SQuaRes (Structured Quartet Research Ensembles) workshops, and we gratefully acknowledge the support of AIM.

{\small

\begin{tabular}{ll}
\begin{tabular}{l}
{\sc Zolt\'an F\" uredi} \\
Alfr\' ed R\' enyi Institute of Mathematics \\
Hungarian Academy of Sciences \\
Re\'{a}ltanoda utca 13-15 \\
H-1053, Budapest, Hungary \\
E-mail:  \texttt{zfuredi@gmail.com}.
\end{tabular}
& \begin{tabular}{l}
{\sc Tao Jiang} \\
Department of Mathematics \\ Miami University \\ Oxford, OH 45056, USA. \\ E-mail: \texttt{jiangt@miamioh.edu}. \\
$\mbox{ }$
\end{tabular} \\ \\
\begin{tabular}{l}
{\sc Alexandr Kostochka} \\
University of Illinois at Urbana--Champaign \\
Urbana, IL 61801 \\
and Sobolev Institute of Mathematics \\
Novosibirsk 630090, Russia. \\
E-mail: \texttt {kostochk@math.uiuc.edu}.
\end{tabular} & \begin{tabular}{l}
{\sc Dhruv Mubayi} \\
Department of Mathematics, Statistics \\
and Computer Science \\
University of Illinois at Chicago \\
Chicago, IL 60607. \\
\texttt{E-mail: mubayi@uic.edu}.
\end{tabular} \\ \\
\begin{tabular}{l}
{\sc Jacques Verstra\"ete} \\
Department of Mathematics \\
University of California at San Diego \\
9500 Gilman Drive, La Jolla, California 92093-0112, USA. \\
E-mail: {\tt jverstra@math.ucsd.edu.}
\end{tabular}
\end{tabular}
}


\begin{thebibliography}{99}

\bibitem{Aronov} B. Aronov, V. Dujmovi\v{c}, P. Morin, A. Ooms, L. da Silveira, More Tur\'{a}n-type theorems for triangles in convex point sets,
Elctronic Journal of Combinatorics, Volume 26, Issue 1 (2019)
Article Number
P1.8.

\bibitem{Brass} P. Bra{\ss}, Tur\'{a}n-type extremal problems for convex geometric hypergraphs. Contemporary
Mathematics, 342, 25--34, 2004.

\bibitem{Brass-Karolyi-Valtr} P. Bra{\ss}, G. K\'{a}rolyi, P. Valtr, A Tur\'{a}n-type extremal theory of convex geometric graphs, Goodman-Pollack Festschrift, Springer 2003, 277--302.

\bibitem{Brass-Rote-Swanepoel} P. Bra{\ss}, G. Rote, K. Swanepoel, Triangles of extremal area or
perimeter in a finite planar point set. Discrete Comp. Geom., 26 (1), 51--58, 2001.

\bibitem{Capoyleas-Pach} V. Capoyleas, J. Pach, A Tur\'{a}n-type theorem for chords of a convex polygon, J. Combin. Theory Ser. B, 56, 9--15.

%\bibitem{Chaya-Perles} C. Keller, M. Perles, On convex geometric graphs with no $k+1$ pairwise disjoint edges. Graphs Combin. 32 (2016), no. 6, 2497--2514.

\bibitem{Erdos} P. Erd\H{o}s, On Sets of Distances of $n$ Points, Amer. Math. Monthly 53 (1946), pp. 248--250.


\bibitem{EK} P. Erd\H os, D. Kleitman, On coloring graphs to maximize the proportion of multicolored
k-edges. J. Combin. Theory 5 (1968) 164--169.

%\bibitem{Erdos-Gallai} P. Erd\H{o}s and T. Gallai,
%On maximal paths and circuits of graphs,
%Acta Math. Acad. Sci. Hungar. {\bf 10} (1959), 337--356.

%\bibitem{Erdos-Sos-Tree} P. Erd\H{o}s, Extremal problems in graph theory
%M. Fiedler (Ed.), Theory of Graphs and its Applications, Academic Press (1965), pp. 29--36.

\bibitem{Furedi} Z. F\"{u}redi, The maximum number of unit distances in a convex $n$-gon, J. Combin. Theory Ser. A, 55 (1990), 316--320.

\bibitem{FJKMV} Z. F\"{u}redi, T. Jiang, A. Kostochka, D. Mubayi, J. Verstraete, Tight paths in convex geometric hypergraphs, https://arxiv.org/abs/1709.01173.

\bibitem{FKMV}  Z. F\"uredi, A. Kostochka, D. Mubayi, J. Verstraete, Ordered and convex geometric trees with linear extremal function, https://arxiv.org/abs/1812.05750.

\bibitem{FJKMV2} Z. F\"{u}redi, T. Jiang, A. Kostochka, D. Mubayi, J. Verstra\" ete, A splitting theorem for ordered hypergraphs, available on arXiv


%\bibitem{FF} P. Frankl, Z. F\"{u}redi, Exact solution of some Tur\'{a}n-type problems. J. Comb. Theory Ser. A 45 (2), 226--262 (1987)

%\bibitem{GKLO} S. Glock, D. K\"{u}hn, A. Lo, D. Osthus, The existence of designs via iterative absorption. Preprint (2017).

\bibitem{Gowers-Long} W. T. Gowers, J. Long, The length of an $s$-increasing sequence of $r$-tuples, https://arxiv.org/abs/1609.08688.


%\bibitem{GRS} A. Gy\'{a}rf\'{a}s, C. Rousseau, R. Schelp, An extremal problem for paths in bipartite graphs. J. Graph Theory 8 (1984), no. 1, 83--95.

\bibitem{Hopf-Pannwitz} H. Hopf and E. Pannwitz: Aufgabe Nr. 167, Jahresbericht d. Deutsch. Math. Verein. 43
(1934), 114.

%\bibitem{Kalai} Personal communication;  https://gilkalai.wordpress.com/page/2/.

%\bibitem{Keevash} P. Keevash, The existence of designs. Preprint (2017).

\bibitem{KTTW} D. Kor\'{a}ndi, G. Tardos, I. Tomon, C. Weidert,
On the TurÃ¡n number of ordered forests,
Electronic Notes in Discrete Mathematics
Volume 61, August 2017, Pages 773-779.

%\bibitem{Kupitz} Y. S. Kupitz, On Pairs of Disjoint Segments in Convex Position in the Plane, Annals
%Discrete Math., 20 (1984), pp. 203--208.

\bibitem{Kupitz-Perles} Y. S. Kupitz, M. Perles, Extremal theory for convex matchings in convex geometric graphs, Discrete Comput Geom. 15, (1996), 195--220.

%\bibitem{Loh} P. Loh, Directed paths: from Ramsey to Ruzsa and Szemer\'{e}di,
%arxiv:1505.07312v2, 2016.

\bibitem{MT}  A. Marcus, G. Tardos, Excluded permutation matrices and the Stanley-Wilf conjecture, Journal
of Combinatorial Theory, Ser. A 107 (2004), 153--160.

\bibitem{PT} J. Pach, G. Tardos, Forbidden paths and cycles in ordered graphs and
matrices, Israel Journal of Mathematics 155 (2006), 359--380.

\bibitem{Pach-Pinchasi} J. Pach, R. Pinchasi, How many unit equilateral triangles can be generated by $n$ points in general position? Amer. Math. Monthly 110 (2003), 100--106.

\bibitem{Perles} M. Perles, unpublished.

\bibitem{Sutherland} J. W. Sutherland, L\"osung der Aufgabe 167,
Jahresbericht Deutsch. Math.-Verein. 45 (1935), 33--35.

\bibitem{T} G. Tardos, Extremal theory of ordered graphs, Proceedings of the International Congress of Mathematics -- 2018, Vol. 3, 3219--3228.

%\bibitem{vanLint-Wilson} J. van Lint, R. Wilson, A course in combinatorics. Second edition. Cambridge University Press, Cambridge, 2001.


\end{thebibliography}
\end{document}